\documentclass[12pt,a4paper,reqno]{amsart} 

\usepackage[T1]{fontenc}       

\usepackage{amsfonts}          

\usepackage{amsmath}

\usepackage{amssymb}

\usepackage{overpic}
\usepackage{euscript} 
\usepackage{eurosym}
\usepackage{comment}
\usepackage{mathrsfs} 
\usepackage{ifthen}
\usepackage{esint} 
\usepackage{color}

\long\def\symbolfootnote[#1]#2{\begingroup%
\def\thefootnote{\fnsymbol{footnote}}\footnote[#1]{#2}\endgroup}

{\qed\vspace{5pt}}

\usepackage{amsthm}

\newtheoremstyle{lause}
{5pt}
{5pt}
{\slshape}
{\parindent}
{\bfseries}
{.}
{.5em}
{}

\theoremstyle{lause}

\newtheoremstyle{maaritelma}
{5pt}
{5pt}
{\rmfamily}
{\parindent}
{\bfseries}
{.}
{.5em}
{}

\theoremstyle{maaritelma}
\newtheoremstyle{lause}
{5pt}
{5pt}
{\slshape}
{\parindent}
{\bfseries}
{.}
{.5em}
{}

\theoremstyle{lause}
\newtheorem{theorem}{Theorem}[section]
\newtheorem{lemma}[theorem]{Lemma}

\newtheorem{corollary}[theorem]{Corollary}

\newtheoremstyle{maaritelma}
{5pt}
{5pt}
{\rmfamily}
{\parindent}
{\bfseries}
{.}
{.5em}
{}

\theoremstyle{maaritelma}
\newtheorem{definition}[theorem]{Definition}

\newtheorem{example}[theorem]{Example}
\newtheorem{remark}[theorem]{Remark}

\DeclareMathOperator*{\essinf}{ess\,inf}

\numberwithin{equation}{section}

\begin{document}

\thispagestyle{empty}

\begin{center}

{\large{\textbf{
On the role of the point at infinity in Deny's principle of positivity of mass for Riesz potentials
}}}

\vspace{18pt}

\textbf{Natalia Zorii}

\vspace{18pt}


\footnotesize{\address{Institute of Mathematics, Academy of Sciences
of Ukraine, Tereshchenkivska~3, 01601,
Kyiv-4, Ukraine\\
natalia.zorii@gmail.com }}

\end{center}

\vspace{12pt}

{\footnotesize{\textbf{Abstract.} First introduced by J.~Deny, the classical principle of positivity of mass states that if $\kappa_\alpha\mu\leqslant\kappa_\alpha\nu$ everywhere on $\mathbb{R}^n$, then $\mu(\mathbb{R}^n)\leqslant\nu(\mathbb{R}^n)$. Here $\mu,\nu$ are positive Radon measures on $\mathbb{R}^n$, $n\geqslant2$, and $\kappa_\alpha\mu$ is the potential of $\mu$ with respect to the Riesz kernel $|x-y|^{\alpha-n}$ of order $\alpha\in(0,2]$, $\alpha<n$. We strengthen Deny's principle by showing that $\mu(\mathbb{R}^n)\leqslant\nu(\mathbb{R}^n)$ still holds even if $\kappa_\alpha\mu\leqslant\kappa_\alpha\nu$ is fulfilled only on a proper subset $A$ of $\mathbb{R}^n$ that is not inner $\alpha$-thin at infinity; and moreover, this condition on $A$ cannot in general be improved. Hence, if $\xi$ is a signed measure on $\mathbb{R}^n$ with $\int1\,d\xi>0$, then $\kappa_\alpha\xi>0$  everywhere on $\mathbb{R}^n$, except for a subset which is inner $\alpha$-thin at infinity. The analysis performed is based on the author's recent theories of inner Riesz balayage and inner Riesz equilibrium measures (Potential Anal., 2022), the inner equilibrium measure being understood in an extended sense where both the energy and the total mass may be infinite.
}}
\symbolfootnote[0]{\quad 2010 Mathematics Subject Classification: Primary 31C15.}
\symbolfootnote[0]{\quad Key words: Principle of positivity of mass for $\alpha$-Riesz potentials, inner $\alpha$-thinness at infinity, inner $\alpha$-Riesz balayage, a generalized concept of inner $\alpha$-Riesz equilibrium measure.}

\vspace{6pt}

\markboth{\emph{Natalia Zorii}} {\emph{On the role of the point at infinity in Deny's principle of positivity of mass for Riesz potentials
}}

\section{A strengthened version of Deny's principle of positivity of mass}\label{intr}

In this paper we shall deal with the theory of potentials on $\mathbb{R}^n$, $n\geqslant2$, with respect to the $\alpha$-Riesz kernel $\kappa_\alpha(x,y):=|x-y|^{\alpha-n}$ of order $\alpha\in(0,2]$, $\alpha<n$, where $|x-y|$ is the Euclidean distance between $x,y\in\mathbb{R}^n$. Denote by $\mathfrak{M}^+$ the cone of all positive Radon measures $\mu$ on $\mathbb{R}^n$ such that the $\alpha$-Riesz {\it potential\/}
\[\kappa_\alpha\mu(x):=\int\kappa_\alpha(x,y)\,d\mu(y)\]
is not identically infinite, which according to \cite[Section~I.3.7]{L} occurs if and only if
\begin{equation}\label{main}
 \int_{|y|>1}\frac{d\mu(y)}{|y|^{n-\alpha}}<\infty.
\end{equation}
Note that the measures in question may be {\it unbounded}, namely with $\mu(\mathbb{R}^n)=+\infty$.

The principle of positivity of mass was first introduced by J.~Deny (see e.g.\ \cite{D2}), and for $\alpha$-Riesz potentials it reads as follows \cite[Theorem~3.11]{FZ}.

\begin{theorem}\label{th-fz}
  For any\/ $\mu,\nu\in\mathfrak{M}^+$ such that
  \begin{equation}\label{FZ1}
  \kappa_\alpha\mu\leqslant\kappa_\alpha\nu\text{ \ everywhere on $\mathbb{R}^n$},
  \end{equation}
  we have\/ $\mu(\mathbb{R}^n)\leqslant\nu(\mathbb{R}^n)$.
\end{theorem}

It is easy to verify that (\ref{FZ1}) can be slightly weakened by replacing `everywhere on $\mathbb{R}^n$' by `nearly everywhere on $\mathbb{R}^n$' (see e.g.\ the author's recent result \cite[Theorem~2.1]{Z-arx-22}, establishing the principle of positivity of mass for potentials with respect to rather general function kernels on locally compact spaces). Recall that a proposition involving a variable point $x\in\mathbb{R}^n$ is said to hold {\it nearly everywhere\/} ({\it n.e.})\ on $A\subset\mathbb{R}^n$ if $c_\alpha(E)=0$, where $E$ is the set of $x\in A$ for which the proposition fails to hold, while $c_\alpha(E)$ is the {\it inner\/ $\alpha$-Riesz capacity\/} of $E$ \cite[Section~II.2.6]{L} (cf.\ Sect.~\ref{sec-prel'} below).

In the present study we shall show that Theorem~\ref{th-fz} still holds even if the assumption $\kappa_\alpha\mu\leqslant\kappa_\alpha\nu$ is fulfilled only on a proper subset $A$ of $\mathbb{R}^n$, which however must be `large enough' in an arbitrarily small neighborhood of $\infty_{\mathbb{R}^n}$, the Alexandroff point
of $\mathbb R^n$. This discovery illustrates a special role of the point at infinity in Riesz potential theory, in particular in the principle of positivity of mass.

To be exact, the following theorem holds true.\footnote{Theorem~\ref{th1} has already found an application to minimum $\alpha$-Riesz energy problems in the presence of external fields, see the author's recent work \cite{Z-Rarx} (Section~4.10, Proof of Theorem~2.13).}

\begin{theorem}\label{th1}
Given\/ $\mu,\nu\in\mathfrak{M}^+$, assume there exists\/ $A\subset\mathbb{R}^n$ which is not inner\/ $\alpha$-thin at infinity, and such that
\begin{equation*}
  \kappa_\alpha\mu\leqslant\kappa_\alpha\nu\text{ \ n.e.\ on $A$}.
  \end{equation*}
Then
\begin{equation*}\label{eq2-th1}
  \mu(\mathbb{R}^n)\leqslant\nu(\mathbb{R}^n).
  \end{equation*}
\end{theorem}

\begin{remark}
The concept of inner $\alpha$-thinness of a set at infinity was introduced in \cite[Definition~2.1]{Z-bal2}.
Referring to Sect.~\ref{sec-prel} for details, at this point we only note that $A\subset\mathbb{R}^n$ {\it is not inner\/ $\alpha$-thin at infinity if and only if\/ $y=0$ is inner\/ $\alpha$-regular for\/ $A^*$, the inverse of\/ $A$ with respect to\/} $\{|x|=1\}$. Such $A$ may in particular be thought of as $\{x_i>q\}$, where $q\in\mathbb{R}$ and $x_i$ is the $i$-coordinate of $x\in\mathbb{R}^n$, or as $Q^c:=\mathbb{R}^n\setminus Q$, $Q\subset\mathbb{R}^n$ being bounded (see also Example~\ref{ex}). {\it If\/ $A$ is not inner\/ $\alpha$-thin at infinity, it is `rather large' at infinity in the sense that then necessarily\/ $c_\alpha(A)=\infty$, whereas the converse is in general not true\/} (see Corollary~\ref{twoC}, cf.\ Example~\ref{exx} for illustration).
\end{remark}

The following theorem shows that Theorem~\ref{th1} is {\it sharp\/} in the sense that the requirement on $A$ of not being $\alpha$-thin at infinity cannot in general be weakened.

\begin{theorem}\label{th-sharp}
  If a set\/ $A\subset\mathbb{R}^n$ is inner\/ $\alpha$-thin at infinity, there exist\/ $\mu_0,\nu_0\in\mathfrak{M}^+$ such that\/ $\kappa_\alpha\mu_0=\kappa_\alpha\nu_0$ n.e.\ on\/ $A$, but nonetheless, $\mu_0(\mathbb{R}^n)>\nu_0(\mathbb{R}^n)$.
\end{theorem}

Nevertheless, Theorem~\ref{th1} remains valid for {\it arbitrary\/} $A\subset\mathbb{R}^n$ once we impose upon the measures $\mu$ and $\nu$ suitable additional requirements (see Theorem~\ref{th-any}).

A measure $\mu\in\mathfrak{M}^+$ is said to be {\it concentrated\/} on $A\subset\mathbb{R}^n$ if $A^c$ is $\mu$-negligible, or equivalently if $A$ is $\mu$-measurable and $\mu=\mu|_A$, $\mu|_A$ being the restriction of $\mu$ to $A$.  We denote by $\mathfrak{M}^+_A$ the cone of all $\mu\in\mathfrak{M}^+$ concentrated on $A$. (For a closed set $A$, a measure $\mu$ belongs to $\mathfrak{M}^+_A$ if and only if its support $S(\mu)$ is contained in $A$.)

A measure $\mu\in\mathfrak M^+$ is said to be {\it $c_\alpha$-absolutely continuous\/} if $\mu(K)=0$ for every compact set $K\subset\mathbb R^n$ with $c_\alpha(K)=0$. This certainly occurs if $\int\kappa_\alpha\mu\,d\mu$ is finite or, more generally, if $\kappa_\alpha\mu$ is locally bounded (but not conversely, see \cite[pp.~134--135]{L}).

\begin{theorem}\label{th-any}
  For any set\/ $A\subset\mathbb{R}^n$ and any\/ $c_\alpha$-absolutely continuous measures\/ $\mu,\nu\in\mathfrak{M}^+_A$ such that\/ $\kappa_\alpha\mu\leqslant\kappa_\alpha\nu$ n.e.\ on\/ $A$, we still have\/ $\mu(\mathbb{R}^n)\leqslant\nu(\mathbb{R}^n)$.
\end{theorem}

\begin{remark}\label{rem-any}
  If $A\cap A_I=\varnothing$, where $A_I$ consists of all inner $\alpha$-irregular points for the set $A$, then the condition of the $c_\alpha$-absolute continuity imposed on $\mu$ and $\nu$, is unnecessary for the validity of Theorem~\ref{th-any}. Namely, {\it for any\/ $A\subset\mathbb{R}^n$ such that\/ $A\cap A_I=\varnothing$, and any\/ $\mu,\nu\in\mathfrak{M}^+_A$ with the property\/ $\kappa_\alpha\mu\leqslant\kappa_\alpha\nu$ n.e.\ on\/ $A$, we still have\/ $\mu(\mathbb{R}^n)\leqslant\nu(\mathbb{R}^n)$}. (See the end of Sect.~\ref{theend} for the proof of this assertion, and Sect.~\ref{subseq1} for the concept of inner $\alpha$-irregular point and relevant results.)
\end{remark}

Let $\mathfrak{M}$ stand for the linear space of all real-val\-ued (signed) Radon measures $\xi$ on $\mathbb R^n$ such that (\ref{main}) is fulfilled with $\mu$ replaced by $|\xi|:=\xi^++\xi^-$, where $\xi^+$ and $\xi^-$ denote the positive and negative parts of $\xi$ in the Hahn--Jor\-dan decomposition. Then the potential $\kappa_\alpha\xi$ of any $\xi\in\mathfrak{M}$ is well defined and finite n.e.\ on $\mathbb{R}^n$ (actually, even everywhere on $\mathbb{R}^n$ except for a polar set, see \cite[Section~III.1.1]{L}).

The next two corollaries follow directly from Theorem~\ref{th1}.

\begin{corollary}
  For any\/ $\xi\in\mathfrak{M}$ with\/ $\xi^+(\mathbb{R}^n)>\xi^-(\mathbb{R}^n)$, $\kappa_\alpha\xi>0$ holds true everywhere on $\mathbb{R}^n$, except for a subset that is inner $\alpha$-thin at infinity.
\end{corollary}

\begin{corollary}
 For any\/ $\xi\in\mathfrak{M}$ with\/ $\xi^+(\mathbb{R}^n)\ne\xi^-(\mathbb{R}^n)$, the set of all\/ $x\in\mathbb{R}^n$ for which\/ $\kappa_\alpha\xi(x)=0$ is inner\/ $\alpha$-thin at infinity.
\end{corollary}

\begin{example}\label{ex} Let $n=3$ and $\alpha=2$. Consider the rotation bodies
\begin{equation}\label{descr}F_j:=\bigl\{x\in\mathbb R^3: \ 0\leqslant x_1<\infty, \
x_2^2+x_3^2\leqslant\varrho_j^2(x_1)\bigr\}, \ j=1,2,\end{equation}
where
\begin{align}
\varrho_1(x_1)&:=x_1^{-s}\text{ \ with\ }s\in[0,\infty),\notag\\
\varrho_2(x_1)&:=\exp(-x_1^s)\text{ \ with\ }s\in(0,\infty).\label{ex2}
\end{align}
By the strengthened principle of positivity of mass (Theorem~\ref{th1}), for any $\mu,\nu\in\mathfrak{M}^+$ such that $\kappa_\alpha\mu\leqslant\kappa_\alpha\nu$ n.e.\ on $F_1$, we have $\mu(\mathbb{R}^3)\leqslant\nu(\mathbb{R}^3)$, the set $F_1$ being not $2$-thin at infinity \cite[Example~2.1]{Z-bal2} (see Figure~\ref{Fig2}).

On the other hand, the set $F_2$ is $2$-thin at infinity \cite[Example~2.1]{Z-bal2} (see Figure~\ref{Fig1}), and hence, according to Theorem~\ref{th-sharp}, there are $\mu_0,\nu_0\in\mathfrak{M}^+$ such that
\[\nu_0(\mathbb{R}^3)<\mu_0(\mathbb{R}^3),\text{ \ though \ }\kappa_\alpha\mu_0=\kappa_\alpha\nu_0\text{\ n.e.\ on $F_2$}.\]
(In fact, any nonzero bounded positive measure concentrated on $F_2^c$ may serve as $\mu_0$, and the Newtonian balayage of this  $\mu_0$ onto $F_2$ then serves as $\nu_0$ \cite[Example~8.8]{Z-bal}.)
\end{example}

\begin{figure}[htbp]
\begin{center}
\vspace{-.8in}
\hspace{-.1in}\includegraphics[width=4.6in]{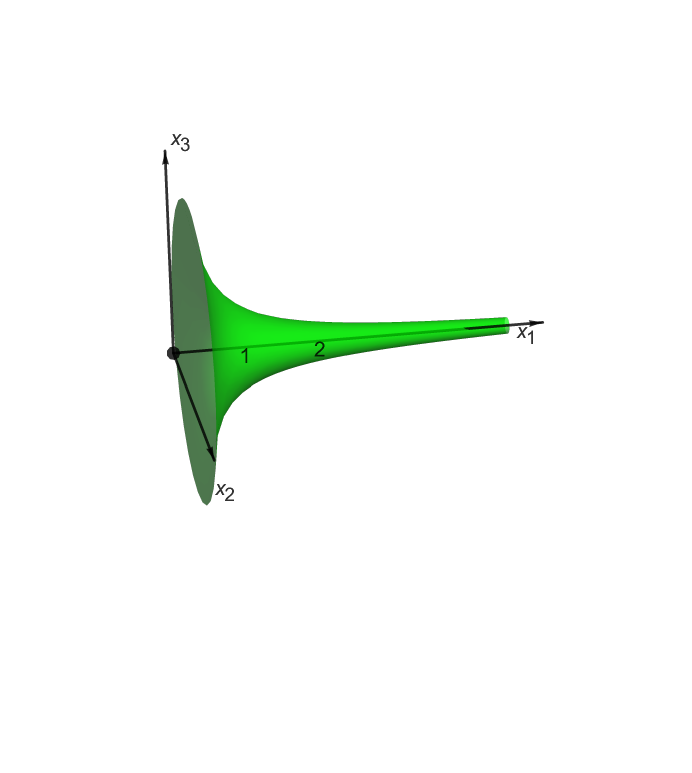}
\vspace{-1.6in}
\caption{The set $F_1$ in Example~\ref{ex} with $\varrho_1(x_1)=1/x_1$.\vspace{-.1in}}
\label{Fig2}
\end{center}
\end{figure}

\begin{figure}[htbp]
\begin{center}
\vspace{-.4in}
\hspace{-1.1in}\includegraphics[width=4.2in]{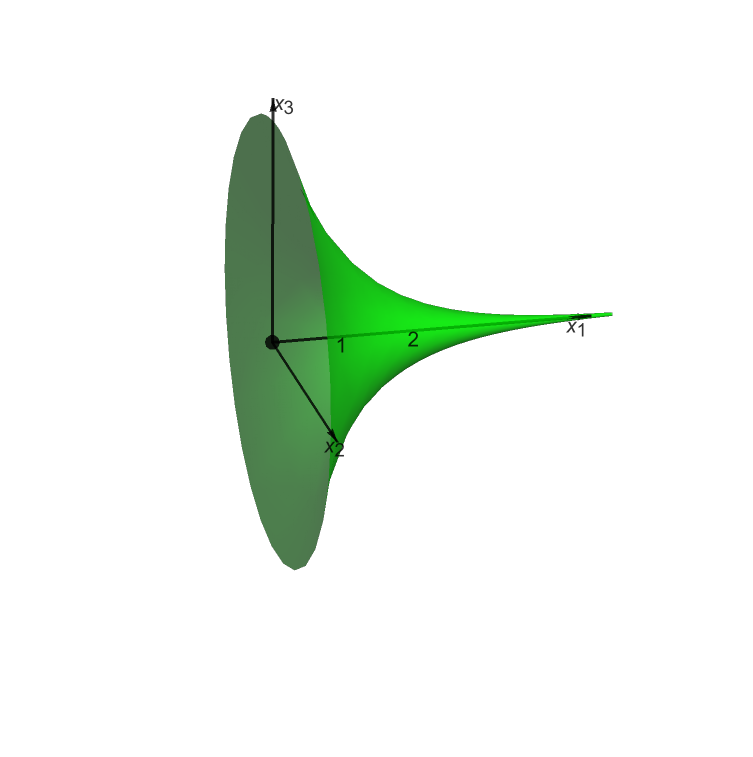}
\vspace{-.8in}
\caption{The set $F_2$ in Example~\ref{ex} with $\varrho_2(x_1)=\exp(-x_1)$.\vspace{-.1in}}
\label{Fig1}
\end{center}
\end{figure}

The proofs of Theorems~\ref{th1}, \ref{th-sharp}, and \ref{th-any} (see Sect.~\ref{sec-proofs}) are based on the theory of inner $\alpha$-Riesz balayage as well as on the theory of inner $\alpha$-Riesz equilibrium measures, both originated in \cite{Z-bal,Z-bal2} (see also \cite{Z-arx1}--\cite{Z-arx} for some further relevant results). The inner $\alpha$-Riesz equilibrium measure is understood in an extended sense where its energy as well as its total mass may be infinite. To make the present study self-con\-tai\-ned, we give a brief summary of  \cite{Z-bal}--\cite{Z-arx} (see Sects.~\ref{subseq1}, \ref{sec-prel}). To begin with, we first review some basic facts of the theory of $\alpha$-Riesz potentials, see~\cite{L}.

\section{Basic facts of the theory of $\alpha$-Riesz potentials}\label{sec-prel'}

In what follows we shall use the notations and conventions introduced in Sect.~\ref{intr}.

Throughout the paper the linear space $\mathfrak{M}$ is meant to be equipped with the (Hausdorff) {\it vague\/} topology of pointwise convergence on the class $C_0(\mathbb R^n)$ of all continuous functions $f:\mathbb R^n\to\mathbb{R}$ of compact support.

The Riesz composition identity $\kappa_\alpha=\kappa_{\alpha/2}\ast\kappa_{\alpha/2}$ \cite[Section~1, Eq.~(12)]{R} (cf.\ also \cite[Eq.~(1.1.12)]{L}) implies that the kernel $\kappa_\alpha$ is {\it strictly positive definite}, which means that for any (signed) $\mu\in\mathfrak{M}$, the $\alpha$-Riesz {\it energy\/} \[\kappa_\alpha(\mu,\mu):=\int\kappa_\alpha(x,y)\,d(\mu\otimes\mu)(x,y)\]
is ${}\geqslant0$ whenever defined, and moreover it is zero only for zero measure. This in turn implies (see e.g.\ \cite[Lemma~3.1.2]{F1}) that all the measures in $\mathfrak{M}$ of {\it finite\/} $\alpha$-Riesz energy form a pre-Hil\-bert space $\mathcal E_\alpha$ with the inner product
\[\langle\mu,\nu\rangle:=\kappa_\alpha(\mu,\nu):=\int\kappa_\alpha(x,y)\,d(\mu\otimes\nu)(x,y)\]
and the energy norm $\|\mu\|:=\sqrt{\kappa_\alpha(\mu,\mu)}$. The (Hausdorff) topology on $\mathcal E_\alpha$ defined by means of this norm is said to be {\it strong}.

By Deny \cite{D1} (for $\alpha=2$, see also H.~Cartan \cite{Ca1}), the cone $\mathcal E^+_\alpha:=\{\mu\in\mathcal E_\alpha:\ \mu\geqslant0\}$ is strongly complete, and the strong topology on $\mathcal E^+_\alpha$ is finer than the (induced) vague topology.\footnote{In B.~Fuglede's terminology \cite{F1}, the Riesz kernel is, therefore, {\it perfect}.} Thus every strong Cauchy net $(\mu_s)\subset\mathcal E^+_\alpha$ converges to the same (unique) limit both strongly and vaguely. This in particular implies that if a set $A\subset\mathbb{R}^n$ is {\it closed\/} (in $\mathbb{R}^n$), then the cone $\mathcal E^+_A:=\mathcal E^+_\alpha\cap\mathfrak{M}^+_A$ is {\it strongly complete}, $\mathfrak{M}^+_A$ being vaguely closed \cite[Section~III.2, Proposition~6]{B2}.

The {\it inner\/ $\alpha$-Riesz capacity\/} $c_\alpha(A)$ of arbitrary $A\subset\mathbb{R}^n$ is defined by the formula\footnote{As usual, the infimum over the empty set is taken to be $+\infty$. We also agree that $1/(+\infty)=0$ and $1/0 = +\infty$.\label{f1}}
\begin{equation}\label{capp}
  c_\alpha(A):=1\bigl/\inf\,\|\mu\|^2,
\end{equation}
the infimum being taken over all $\mu\in\mathcal E^+_A$ with $\mu(\mathbb{R}^n)=1$. If $c_\alpha(A)<\infty$, there exists the unique solution $\gamma_A$ to the minimum $\alpha$-Riesz energy problem over the class of all $\nu\in\mathcal E^+_\alpha$ such that $\kappa_\alpha\nu\geqslant1$ n.e.\ on $A$ (see e.g.\ \cite[Theorem~4.1]{F1}). This $\gamma_A$ is said to be {\it the inner\/ $\alpha$-Riesz equilibrium measure\/} for $A$,\footnote{We also refer to the author's recent work \cite{Z-arx-22} providing a number of alternative characterizations of the inner capacity $c_\alpha(A)$ and the inner equilibrium measure $\gamma_A$, the results in \cite{Z-arx-22} being actually valid even for quite a large class of general function kernels on locally compact spaces.} and it satisfies the relations
\begin{gather}\gamma_A(\mathbb{R}^n)=\kappa_\alpha(\gamma_A,\gamma_A)=c_\alpha(A)\in[0,\infty),\label{fcap1}\\
\kappa_\alpha\gamma_A=1\text{ \ n.e.\ on $A$}.\label{fcap2}\end{gather}

For closed $A\subset\mathbb{R}^n$, the inner $\alpha$-Riesz equilibrium measure $\gamma_A$ can alternatively be characterized as the only measure in $\mathcal E^+_A$ satisfying (\ref{fcap2}). But if $A$ is not closed, $\gamma_A$ may not be concentrated on the set $A$ itself, but on $\overline{A}$, the closure of $A$ in $\mathbb{R}^n$; and moreover, there is in general no $\nu\in\mathcal E^+_A$ having the property $\kappa_\alpha\nu=1$ n.e.\ on $A$. Regarding the latter, see \cite{Z-arx-22} (Theorem~1.1(e) and footnote~3).

For reasons of homogeneity,
\[c_\alpha(A)=0\iff\mathcal E^+_A=\{0\}\iff\mathcal E^+_K=\{0\}\text{ \ for every compact $K\subset A$}\]
 (cf.\ \cite[Lemma~2.3.1]{F1}). This in turn implies that {\it for any measure\/ $\mu\in\mathcal E^+_\alpha$ and any\/ $\mu$-mea\-sur\-ab\-le set\/ $A\subset\mathbb{R}^n$ with\/ $c_\alpha(A)=0$, $A$ is\/ $\mu$-neg\-lig\-ible}.

Along with the perfectness of the $\alpha$-Riesz kernels, the following Theorems~\ref{3.9} and \ref{dom} were crucial to the development of the theory of inner balayage and that of inner equilibrium measures $\gamma_A$, $\gamma_A$ being understood in an extended sense where both $\kappa_\alpha(\gamma_A,\gamma_A)$ and $\gamma_A(\mathbb{R}^n)$ may be infinite (see \cite{Z-bal}--\cite{Z-arx}, cf.\ Sects.~\ref{subseq1}, \ref{sec-prel} below).

\begin{theorem}\label{3.9}If a net\/ $(\kappa_\alpha\mu_s)_{s\in S}$, where\/ $(\mu_s)_{s\in S}\subset\mathfrak M^+$, increases  pointwise on\/ $\mathbb R^n$, and is majorized by\/ $\kappa_\alpha\nu$ for some\/ $\nu\in\mathfrak M^+$, then there exists\/ $\mu_0\in\mathfrak M^+$ such that\/ $\kappa_\alpha\mu_s\uparrow\kappa_\alpha\mu_0$ pointwise on\/ $\mathbb R^n$ and\/ $\mu_s\to\mu_0$ vaguely\/ {\rm(}as\/ $s$ ranges through\/ $S$\rm{)}.
\end{theorem}

\begin{proof}
 If $(\mu_s)_{s\in S}$ is a sequence, Theorem~\ref{3.9} is, in fact, \cite[Theorem~3.9]{L} (cf.\ also \cite{BrCh,Ca1}). The proof of \cite[Theorem~3.9]{L} can be generalized to the case where $(\mu_s)_{s\in S}$ is a net, by use of \cite[Appendix~VIII, Theorem~2]{Doob} and \cite[Section~IV.1, Theorem~1]{B2}.
\end{proof}

The property of the $\alpha$-Riesz kernels (of order $\alpha\in(0,2]$, $\alpha<n$), presented in the following theorem (see \cite[Theorems~1.27, 1.29]{L}), is known in the literature as {\it the complete maximum principle}; for $q=0$, it is also called  {\it the domination principle}, and for $\nu=0$, {\it the Frostman maximum principle}.

\begin{theorem}\label{dom}If\/ $\kappa_\alpha\mu\leqslant\kappa_\alpha\nu+q$ holds true $\mu$-a.e.\ {\rm(}$\mu$-almost everywhere\/{\rm)}, where\/ $\mu\in\mathcal E^+_\alpha$, $\nu\in\mathfrak M^+$, and\/ $q\in[0,\infty)$, then the same inequality is fulfilled on all of\/ $\mathbb R^n$.\end{theorem}

\section{Basic facts of the theory of inner $\alpha$-Riesz balayage}\label{subseq1}

The theory of inner $\alpha$-Riesz balayage of arbitrary $\mu\in\mathfrak{M}^+$ to arbitrary $A\subset\mathbb{R}^n$, originated by the author in \cite{Z-bal} (for $\alpha=2$, see the pioneering paper by Cartan \cite{Ca2}), has recently found a further development in \cite{Z-bal2}--\cite{Z-arx}. The present section as well as Sect.~\ref{sec-prel} describes some basic facts from \cite{Z-bal}--\cite{Z-arx}, useful for the understanding of the results of the current study and the methods applied.

Assume for a moment that a set $A:=F$ is {\it closed}, and that a measure $\mu:=\sigma$ is of {\it finite} energy, i.e.\ $\sigma\in\mathcal E^+_\alpha$. Based on the facts that the $\alpha$-Riesz kernel is perfect and satisfies the domination principle, one can prove by generalizing the classical Gauss variational method (see \cite{C0,Ca2}, cf.\ also \cite[Section~IV.5.23]{L}) that there exists $\sigma^F\in\mathcal E^+_F$ uniquely determined within $\mathcal E^+_F$ by the equality $\kappa_\alpha\sigma^F=\kappa_\alpha\sigma$ n.e.\ on $F$.
This $\sigma^F$ is said to be the $\alpha$-Riesz {\it balayage\/} of $\sigma\in\mathcal E^+_\alpha$ onto (closed) $F$, and it can alternatively be characterized as the orthogonal projection of $\sigma$ in the pre-Hil\-bert space $\mathcal E_\alpha$ onto the convex, strongly complete cone $\mathcal E^+_F$ (see Sect.~\ref{sec-prel'});
 that is,\footnote{Concerning the orthogonal projection in a pre-Hilbert space, see e.g.\ \cite[Theorem~1.12.3]{E2}.}
\begin{equation}\label{pr-cl}\|\sigma-\sigma^F\|=\min_{\nu\in\mathcal{E}^+_F}\,\|\sigma-\nu\|.\end{equation}

However, if $A$ is not closed, or if $\mu$ is of infinite energy, then there is in general no measure $\nu$ which would be uniquely determined within $\mathfrak{M}^+_A$ by the equality $\kappa_\alpha\nu=\kappa_\alpha\mu$ n.e.\ on $A$ (see e.g.\ Remark~\ref{rem-2}). Nevertheless, a substantial theory of {\it inner\/} $\alpha$-Riesz balayage of arbitrary $\mu\in\mathfrak{M}^+$ to arbitrary $A\subset\mathbb{R}^n$ was developed \cite{Z-bal}--\cite{Z-arx-22}, and this was performed by means of several alternative approaches described below.\footnote{The {\it outer\/} $\alpha$-Riesz balayage was investigated by J.~Bliedtner and W.~Hansen \cite{BH} in the general framework of balayage spaces. See also N.S.~Landkof \cite[Section~V.1.2]{L}, where, however, certain restrictions were imposed upon $A$ and $\mu$, e.g.\ that $A\subset\mathbb{R}^n$ be Borel while $\mu\in\mathfrak{M}^+$ be bounded.}

Given arbitrary $\mu\in\mathfrak{M}^+$ and $A\subset\mathbb{R}^n$, denote
\begin{equation}\label{gamma}
 \Gamma_{A,\mu}:=\bigl\{\nu\in\mathfrak{M}^+: \ \kappa\nu\geqslant\kappa\mu\text{ \ n.e.\ on $A$}\bigr\}.
\end{equation}
The class $\Gamma_{A,\mu}$ is obviously nonempty, for $\mu\in\Gamma_{A,\mu}$, and it is convex, the latter being clear from the following  strengthened version of countable subadditivity for inner capacity (see \cite[p.~253]{Ca2}, \cite[p.~158, Remark]{F1}; compare with \cite[Section~II.2.6]{L}).

\begin{lemma}\label{str}
For arbitrary\/ $A\subset\mathbb{R}^n$ and Borel\/ $B_j\subset\mathbb{R}^n$, $j\in\mathbb N$,
\[c_\alpha\Bigl(\bigcup_{j\in\mathbb N}\,A\cap B_j\Bigr)\leqslant\sum_{j\in\mathbb N}\,c_\alpha(A\cap B_j).\]\end{lemma}

\begin{definition}\label{def-bal} The {\it inner balayage\/} $\mu^A$ of $\mu\in\mathfrak{M}^+$ to $A\subset\mathbb{R}^n$ is defined as the measure of minimum potential in the class $\Gamma_{A,\mu}$, that is, $\mu^A\in\Gamma_{A,\mu}$ and
\begin{equation}\label{e-d}\kappa_\alpha\mu^A=\min_{\nu\in\Gamma_{A,\mu}}\,\kappa_\alpha\nu\text{ \ on $\mathbb R^n$}.\end{equation}
\end{definition}

This definition is in agreement with Cartan's classical concept of inner Newtonian balayage (cf.\ \cite[Section~19, Theorem~1]{Ca2}). Nonetheless, the results presented below are largely new even for the Newtonian kernel $|x-y|^{2-n}$ on $\mathbb R^n$, $n\geqslant3$.

Denote by $\mathcal{E}'_A$ the closure of $\mathcal{E}^+_A$ in the strong topology on $\mathcal{E}_\alpha^+$. The class $\mathcal{E}'_A$ is convex, for so is $\mathcal{E}^+_A$; and it is strongly complete, being a strongly closed subset of the strongly complete cone $\mathcal{E}^+_\alpha$ (cf.\ Sect.~\ref{sec-prel'}).

\begin{theorem}\label{th-bal}Given arbitrary\/ $\mu\in\mathfrak{M}^+$ and\/ $A\subset\mathbb{R}^n$, the inner balayage\/ $\mu^A$, introduced by Definition\/~{\rm\ref{def-bal}}, exists and is unique. Furthermore,\footnote{Relation~(\ref{ineq1}) actually holds true everywhere on $A^r$, $A^r\subset\overline{A}$ being the set of all inner $\alpha$-re\-gu\-lar points for $A$ (see (\ref{reg-pot}), cf.\ also (\ref{KE})).}
\begin{align}\label{ineq1}\kappa_\alpha\mu^A&=\kappa_\alpha\mu\text{ \ n.e.\ on\ }A,\\
\label{ineq2}\kappa_\alpha\mu^A&\leqslant\kappa_\alpha\mu\text{ \ on\ }\mathbb R^n.
\end{align}
The inner balayage\/ $\mu^A$ can alternatively be characterized by means of either of the following\/ {\rm(}equivalent\/{\rm)} assertions:
\begin{itemize}
  \item[{\rm(a)}] $\mu^A$ is the unique measure in\/ $\mathfrak{M}^+$ satisfying the symmetry relation
  \begin{equation*}\label{eq-sym}
    \kappa_\alpha(\mu^A,\sigma)=\kappa_\alpha(\sigma^A,\mu)\text{ \ for all\/ $\sigma\in\mathcal{E}_\alpha^+$},
  \end{equation*}
  where\/ $\sigma^A$ denotes the only measure in\/ $\mathcal{E}'_A$ with\/ $\kappa_\alpha\sigma^A=\kappa_\alpha\sigma$ n.e.\ on\/ $A$.\footnote{For any $\sigma\in\mathcal{E}^+_\alpha$ and any $A\subset\mathbb{R}^n$, the measure $\sigma^A\in\mathcal{E}'_A$ having the property $\kappa_\alpha\sigma^A=\kappa_\alpha\sigma$ n.e.\ on $A$, exists and is unique. It is in fact the orthogonal projection of $\sigma$ in the pre-Hil\-bert space $\mathcal{E}_\alpha$ onto the convex, strongly complete cone $\mathcal{E}'_A$; that is (compare with (\ref{pr-cl})),
  \[\|\sigma-\sigma^A\|=\min_{\nu\in\mathcal{E}'_A}\,\|\sigma-\nu\|.\]
  Alternatively, $\sigma^A$ is uniquely characterized within $\mathfrak{M}^+$ by the extremal property (\ref{e-d}) with $\mu:=\sigma$.\label{Fo1}}
  \item[{\rm(b)}] $\mu^A$ is the unique measure in\/ $\mathfrak{M}^+$ satisfying either of the two limit relations
  \begin{gather*}\mu_j^A\to\mu^A\text{ \ vaguely in\/ $\mathfrak{M}^+$ as\/ $j\to\infty$},\\
\kappa_\alpha\mu_j^A\uparrow\kappa_\alpha\mu^A\text{ \ pointwise on\/ $\mathbb R^n$ as\/ $j\to\infty$},
\end{gather*}
where\/ $(\mu_j)\subset\mathcal{E}^+_\alpha$ is an arbitrary sequence having the property\/\footnote{Such $\mu_j\in\mathcal{E}^+_\alpha$, $j\in\mathbb N$, do exist; they can be defined, for instance, by means of the formula
\[\kappa_\alpha\mu_j:=\min\,\bigl\{\kappa_\alpha\mu,\,j\kappa_\alpha\lambda\bigr\},\]
$\lambda\in\mathcal{E}^+_\alpha$ being fixed (see e.g.\ \cite[p.~272]{L} or \cite[p.~257, footnote]{Ca2}). Here we have used the fact that for any $\mu_1,\mu_2\in\mathfrak{M}^+$, there is $\mu_0\in\mathfrak{M}^+$ such that $\kappa_\alpha\mu_0:=\min\,\{\kappa_\alpha\mu_1,\,\kappa_\alpha\mu_2\}$ \cite[Theorem~1.31]{L}.}
\begin{equation}\label{eq-mon}\kappa_\alpha\mu_j\uparrow\kappa_\alpha\mu\text{ \ pointwise on\/ $\mathbb{R}^n$ as\/ $j\to\infty$},\end{equation}
whereas\/ $\mu_j^A$ denotes the only measure in\/ $\mathcal{E}'_A$ with\/ $\kappa_\alpha\mu_j^A=\kappa_\alpha\mu_j$ n.e.\ on\/ $A$. {\rm(Regarding the existence and uniqueness of this $\mu_j^A$, see footnote~\ref{Fo1})}.
\end{itemize}
\end{theorem}

\begin{remark}\label{rem-1}
  The inner balayage $\mu^A$ is in general {\it not\/} concentrated on the set $A$ itself, but on its closure $\overline{A}$, and this occurs even for the Newtonian kernel on $\mathbb R^n$, $n\geqslant3$, and $A:=B_r:=\{|x|<r\}$, $r\in(0,\infty)$. Indeed, for any $\mu\in\mathfrak M^+_{\overline{B}_r^c}$, we have $S(\mu^{B_r})=\{|x|=r\}$ (see \cite[Theorems~4.1, 5.1]{Z-bal2}), and hence actually
  $S(\mu^{B_r})\cap B_r=\varnothing$.
\end{remark}

\begin{remark}\label{rem-2}Assume for a moment that $\mu\in\mathcal E^+_\alpha$. As noted in Theorem~\ref{th-bal} and footnote~\ref{Fo1},  the inner balayage $\mu^A$ is then the only measure in $\mathcal{E}'_A$ satisfying (\ref{ineq1}). This in turn implies that there is in general no $\nu\in\mathcal E^+_A$ with $\kappa_\alpha\nu=\kappa_\alpha\mu$ n.e.\ on $A$. Indeed, if there were such $\nu$, then it would necessarily serve as $\mu^A$; which however is in general impossible, for $\mu^A$ may not be concentrated on $A$ (see Remark~\ref{rem-1}).\end{remark}

\begin{corollary}\label{cor-bal}For any\/ $\mu\in\mathfrak M^+$ and any\/ $A\subset\mathbb R^n$, the inner balayage\/ $\mu^A$ is of minimum total mass in the class\/ $\Gamma_{A,\mu}$, that is,
\begin{equation}\label{eq-t-m}\mu^A(\mathbb R^n)=\min_{\nu\in\Gamma_{A,\mu}}\,\nu(\mathbb R^n).\end{equation}\end{corollary}

\begin{proof}
  Since $\mu^A\in\Gamma_{A,\mu}$, we only need to show that
  $\mu^A(\mathbb R^n)\leqslant\nu(\mathbb R^n)$  for all $\nu\in\Gamma_{A,\mu}$, which however follows directly from definition (\ref{gamma}) by use of
  the (classical) principle of positivity of mass (see Theorem~\ref{th-fz}).
\end{proof}

\begin{remark}\label{t-m-nonun} However, the extremal property (\ref{eq-t-m}) cannot serve as an alternative characterization of inner balayage, for it does not determine $\mu^A$ uniquely within $\Gamma_{A,\mu}$. Indeed, consider a closed proper subset $A$ of $\mathbb R^n$ that is not $\alpha$-thin at infinity (take, for instance, $A:=\{|x|\geqslant 1\}$). Then for any $\mu\in\mathfrak M^+_{A^c}$,
\begin{equation}\label{eq-t-m1}\mu^A\ne\mu\text{ \ and \ }\mu^A(\mathbb R^n)=\mu(\mathbb R^n),\end{equation}
the former relation being obvious, and the latter following from Theorem~\ref{th-th}(iii).
Noting that $\mu,\mu^A\in\Gamma_{A,\mu}$ while $\Gamma_{A,\mu}$ is convex, we conclude by combining (\ref{eq-t-m}) with (\ref{eq-t-m1}) that there are actually in $\Gamma_{A,\mu}$ infinitely many measures of minimum total mass, for so is every measure of the form $a\mu+b\mu^A$, where $a,b\in[0,1]$ and $a+b=1$.
\end{remark}

\begin{corollary}\label{cor-sym}Given\/ $\mu\in\mathfrak M^+$ and\/ $A\subset\mathbb R^n$,
\begin{equation}\label{alt}
\kappa_\alpha(\mu^A,\nu)=\kappa_\alpha(\mu,\nu^A)\text{ \ for all\/ $\nu\in\mathfrak M^+$}.
\end{equation}
\end{corollary}

\begin{proof}
   Fix $\mu,\nu\in\mathfrak M^+$, and choose a sequence $(\mu_j)\subset\mathcal{E}^+_\alpha$ satisfying (\ref{eq-mon}). By Theorem~\ref{th-bal} (see (a) and (b)), $(\kappa_\alpha\mu_j^A)$ increases pointwise on $\mathbb R^n$ to $\kappa_\alpha\mu^A$, whereas
 \[\int\kappa_\alpha\mu_j^A\,d\nu=\int\kappa_\alpha\mu_j\,d\nu^A\text{ \ for all\ $j$}.\]
 Applying the monotone convergence theorem \cite[Section~IV.1, Theorem~3]{B2} to each of these two integrals, we obtain (\ref{alt}).
\end{proof}

\begin{remark}It follows from Theorem~\ref{th-bal}(a) and  Corollary~\ref{cor-sym} that, if for a given $\mu\in\mathfrak M^+$, there exists $\zeta\in\mathfrak M^+$ having the property
\[\kappa_\alpha(\zeta,\nu)=\kappa_\alpha(\mu,\nu^A)\text{ \  for all $\nu\in\mathfrak M^+$,}\]
then necessarily $\zeta=\mu^A$. Actually, this characteristic property of the inner balayage $\mu^A$ needs only to be verified for certain countably many $\nu_j\in\mathcal E^+_\alpha$ that are independent of the choice of $\mu\in\mathfrak M^+$.\footnote{This result has recently been extended to inner balayage on a locally compact space, see \cite{Z-arx}.} This is implied by the fact that there are countably many $\nu_j\in\mathcal E^+_\alpha$ whose potentials $\kappa_\alpha\nu_j$ form a dense subset of $C_0(\mathbb R^n)$ (\cite[Lemmas~3.1, 3.2]{Z-bal2}).
\end{remark}

Given $A\subset\mathbb R^n$, denote by $\mathfrak C_A$ the upward directed set of all compact subsets $K$ of $A$, where $K_1\leqslant K_2$ if and only if $K_1\subset K_2$. If a net $(x_K)_{K\in\mathfrak C_A}\subset Y$ converges to $x_0\in Y$, $Y$ being a topological space, then we shall indicate this fact by writing
\begin{equation*}\label{abr}x_K\to x_0\text{ \ in $Y$ as $K\uparrow A$}.\end{equation*}

The following theorem (cf.\ \cite[Theorem~4.5]{Z-bal}), analyzing the convergence of inner swept measures and their potentials under the exhaustion of $A\subset\mathbb R^n$ by compact subsets $K\subset A$, justifies the term `inner' balayage.

\begin{theorem}\label{bal-cont}
  For any\/ $\mu\in\mathfrak M^+$ and any\/ $A\subset\mathbb R^n$,
  \begin{gather*}
    \mu^K\to\mu^A\text{ \ vaguely in\/ $\mathfrak M^+$ as\/ $K\uparrow A$},\\
    \kappa_\alpha\mu^K\uparrow\kappa_\alpha\mu^A\text{ \ pointwise on\/ $\mathbb R^n$ as\/ $K\uparrow A$}.
  \end{gather*}
  If now\/ $\mu\in\mathcal E^+_\alpha$, then also
  \[\mu^K\to\mu^A\text{ \ strongly in\/ $\mathcal E^+_\alpha$ as\/ $K\uparrow A$}.\]
\end{theorem}

A point $y\in\mathbb R^n$ is said to be {\it inner\/ $\alpha$-re\-gu\-lar\/} for $A$ if $\varepsilon_y=(\varepsilon_y)^A=:\varepsilon_y^A$, $\varepsilon_y$ being the unit Dirac measure at $y$;\footnote{$\varepsilon_y^A$ is said to be the (fractional) {\it inner\/ $\alpha$-har\-mon\-ic measure\/} of $A\subset\mathbb R^n$ at $y\in\mathbb R^n$. Being a natural generalization of the classical concept of ($2$-)har\-mo\-nic measure \cite{BH,KB0,KB,L}, $\varepsilon_y^A$ serves as the main tool in solving the generalized Dirichlet problem for $\alpha$-har\-mon\-ic functions. Besides, due to the integral representation formula
$\mu^A=\int\varepsilon_y^A\,d\mu(y)$ \cite[Theorem~5.1]{Z-bal2},
the inner $\alpha$-har\-mon\-ic measure $\varepsilon_y^A$ is a powerful tool in the investigation of the inner balayage $\mu^A$ for arbitrary $\mu$ (see~\cite{Z-bal2}).}
the set of all those $y$ is denoted by $A^r$. Then $A^r\subset\overline{A}$, since obviously $\varepsilon_x^A\in\mathcal E^+_\alpha$ for all $x\notin\overline{A}$.
The other points of $\overline{A}$, i.e.
\[y\in\overline{A}\setminus A^r=:A_I,\]
are said to be {\it inner\/ $\alpha$-ir\-reg\-ul\-ar\/} for $A$. As seen from (\ref{alt}) with $\nu:=\varepsilon_y$,
\begin{equation}\label{reg-pot}y\in A^r\iff\kappa_\alpha\mu^A(y)=\kappa_\alpha\mu(y)\text{ \ for all\ }\mu\in\mathfrak M^+.\end{equation}
It is also worth noting that for any $A\subset\mathbb R^n$, the sets $A^r$ and $A_I$ are Borel measurable \cite[Theorem~5.2]{Z-bal2}, and hence capacitable.

By the Wiener type criterion \cite[Theorem~6.4]{Z-bal},
\begin{equation}\label{w}y\notin A^r\iff\sum_{j\in\mathbb N}\,\frac{c_\alpha(A_j)}{q^{j(n-\alpha)}}<\infty,\end{equation}
where $q\in(0,1)$ and $A_j:=A\cap\{x\in\mathbb R^n:\ q^{j+1}<|x-y|\leqslant q^j\}$, while by the Kel\-logg--Ev\-ans type theorem \cite[Theorem~6.6]{Z-bal},\footnote{Observe that both (\ref{w}) and (\ref{KE}) refer to {\it inner\/} capacity; compare with the Kell\-ogg--Ev\-ans and Wiener type theorems established for {\it outer\/} balayage (see e.g.\ \cite{BH,Br,Ca2,Doob}). Regarding (\ref{KE}), we also note that the whole set $A_I$ may be of nonzero capacity \cite[Section~V.4.12]{L}.}
\begin{equation}\label{KE}c_\alpha(A\cap A_I)=0.\end{equation}
Relation (\ref{w}) implies, in particular, that $A_I\subset\partial A$, where $\partial A$ denotes the boundary of $A$ in the Euclidean topology on $\mathbb R^n$.

\section{Basic facts of the theory of inner $\alpha$-Riesz equilibrium measures}\label{sec-prel}

This section reviews some basic facts of the theory of inner $\alpha$-Riesz equilibrium measures, developed in \cite{Z-bal,Z-bal2,Z-arx-22}. The inner equilibrium measure $\gamma_A$ of $A\subset\mathbb{R}^n$ is understood in an extended sense where its energy $\kappa_\alpha(\gamma_A,\gamma_A)$ as well as its total mass $\gamma_A(\mathbb{R}^n)$ may be infinite (compare with (\ref{fcap1})).

For arbitrary $A\subset\mathbb{R}^n$, define
\begin{equation}\Gamma_A:=\bigl\{\nu\in\mathfrak M^+:\ \kappa_\alpha\nu\geqslant1\text{ \ n.e.\ on $A$}\bigr\}.\label{def-in}\end{equation}

\begin{definition}\label{def-eq}A measure\/ $\gamma_A$ is said to be the {\it inner\/ $\alpha$-Riesz equilibrium measure\/} of $A\subset\mathbb{R}^n$ if it is of minimum potential in $\Gamma_A$, that is, if $\gamma_A\in\Gamma_A$ and\footnote{In view of (permanent) assumption (\ref{main}), the inner $\alpha$-Riesz equilibrium measure $\gamma_A$ does not exist if there is no $\nu\in\mathfrak M^+$ with $\kappa_\alpha\nu\geqslant1$ n.e.\ on $A$. This implication can actually be reversed, and hence $\gamma_A$ exists if and only if $\Gamma_A\ne\varnothing$, see Theorem~\ref{th-th}(ii).}
\[\kappa_\alpha\gamma_A=\min_{\nu\in\Gamma_A}\,\kappa_\alpha\nu\text{ \ on $\mathbb R^n$}.\]
\end{definition}

It follows easily by use of \cite[Theorem~1.12]{L} that the inner $\alpha$-Riesz equilibrium measure $\gamma_A$ is {\it unique\/} (if it exists). Theorem~\ref{th-th} below provides a number of equivalent conditions, each of which is necessary and sufficient for the existence of this $\gamma_A$.
The approach applied is based on the close interaction between the concept of inner equilibrium measure and that of inner balayage, described by means of equality (\ref{har-eq}) with the Kelvin transformation involved.

For every $y\in\mathbb R^n$, define the inversion $J_y$ with respect to the sphere $S(y,1):=\{x\in\mathbb R^n:\ |x-y|=1\}$ mapping each point $x\ne y$ to the point $x^*=J_y(x)$ on the ray through $x$ issuing from $y$ which is uniquely determined by
\[|x-y|\cdot|x^*-y|=1.\]
This is a homeomorphism of $\mathbb R^n\setminus\{y\}$ onto itself having the property
\begin{equation}\label{inv}|x^*-z^*|=\frac{|x-z|}{|x-y||z-y|}\text{ \ for all\ }x,z\in\mathbb R^n\setminus\{y\}.\end{equation}
If now $A\subset\mathbb R^n$ is given, then for any $q\in(0,1)$ and any $A_j$ appearing in the Wiener type criterion (\ref{w}) of inner $\alpha$-regularity,
\begin{equation*}\label{est}q^{-2j(n-\alpha)}c_\alpha(A_j)\leqslant c_\alpha(A_j^*)\leqslant q^{-(2j+2)(n-\alpha)}c_\alpha(A_j), \ j\in\mathbb N,\end{equation*}
where $A_j^*:=J_y(A_j)$. This follows from (\ref{inv}) by use of \cite[Remark to Theorem~2.9]{L}.

For every $\nu\in\mathfrak M^+$ with $\nu(\{y\})=0$, define the {\it Kelvin transform\/}
$\nu^*=\mathcal K_y\nu\in\mathfrak M^+$ by means of the formula (see \cite[Section~IV.5.19]{L})
\begin{equation}\label{kelv-m}d\nu^*(x^*)=|x-y|^{\alpha-n}\,d\nu(x),\text{ \ where\ }x^*=J_y(x)\in\mathbb R^n.\end{equation}
Noting that $(\nu^*)^*=\nu$, we derive from (\ref{inv}) and (\ref{kelv-m}) that
\begin{gather}\nu(\mathbb R^n)=\kappa_\alpha\nu^*(y),\notag\\
\kappa_\alpha(\nu^*,\nu^*)=\kappa_\alpha(\nu,\nu),\label{k1}\\
\kappa_\alpha\nu^*(x^*)=|x-y|^{n-\alpha}\kappa_\alpha\nu(x)\text{ \ for all\ }x^*\in\mathbb R^n.\notag\end{gather}

The proof of the following theorem is based on the theory of inner balayage, reviewed in Sect.~\ref{subseq1}, as well as on the ab\-ove-quo\-ted elementary properties of the inversion and the Kelvin transformation (see \cite[Theorems~2.1, 2.2, 5.1]{Z-bal2} for details).

\begin{theorem}\label{th-th}For arbitrary\/ $A\subset\mathbb R^n$, the following\/ {\rm(i)--(v)} are equivalent.
\begin{itemize}
\item[{\rm (i)}]There exists the inner\/ $\alpha$-Riesz equilibrium measure\/ $\gamma_A$ for\/ $A$, uniquely determined by Definition\/~{\rm\ref{def-eq}}.
\item[{\rm (ii)}]There exists\/ $\nu\in\mathfrak M^+$ with
\[\essinf_{x\in A}\,\kappa_\alpha\nu(x)>0,\]
where\/ $x$ ranges over all of\/ $A$ except for a subset of inner capacity zero.
\item[{\rm (iii)}]There exists\/ $\mu\in\mathfrak M^+$ having the property\/\footnote{In general, $\mu^Q(\mathbb R^n)\leqslant\mu(\mathbb R^n)$, where $\mu\in\mathfrak M^+$ and $Q\subset\mathbb R^n$ are arbitrary. This follows e.g.\ from (\ref{ineq2}) by use of the (classical) principle of positivity of mass (see Theorem~\ref{th-fz}).}
\[\mu^A(\mathbb R^n)<\mu(\mathbb R^n).\]
\item[{\rm (iv)}]For some\/ {\rm(}equivalently, every\/{\rm)} $y\in\mathbb R^n$,
\begin{equation*}\label{iii}\sum_{j\in\mathbb N}\,\frac{c_\alpha(A_j)}{q^{j(n-\alpha)}}<\infty,\end{equation*}
where\/ $q\in(1,\infty)$ and\/ $A_j:=A\cap\{x\in\mathbb R^n:\ q^j\leqslant|x-y|<q^{j+1}\}$.
\item[{\rm (v)}]For some\/ {\rm(}equivalently, every\/{\rm)} $y\in\mathbb R^n$, the inner\/ $\alpha$-har\-mo\-nic measure\/ $\varepsilon_y^{A_y^*}$ is\/ $c_\alpha$-ab\-sol\-ut\-ely continuous.
\end{itemize}

If these\/ {\rm(i)--(v)} hold true, then for every\/ $y\in\mathbb R^n$, the Kelvin transform\/ $(\gamma_A)^*=\mathcal K_y\gamma_A$ of the inner equilibrium measure\/ $\gamma_A$ for the set\/ $A$ is actually the inner\/ $\alpha$-har\-mo\-nic measure\/ $\varepsilon_y^{A_y^*}$ for the inverse\/ $A_y^*=J_y(A)$. That is,
\begin{equation}\label{har-eq}\varepsilon_y^{A_y^*}=(\gamma_A)^*.\end{equation}
\end{theorem}

\begin{definition}[{\rm see \cite[Definition~2.1]{Z-bal2}}]\label{def-th} $A\subset\mathbb R^n$ is said to be {\it inner\/ $\alpha$-thin
at infinity\/} if any of the (equivalent) assertions (i)--(v) in Theorem~\ref{th-th} holds true. Thus $A$ is inner $\alpha$-thin at infinity if and only if some (equivalently, every) point $y\in\mathbb R^n$ is inner $\alpha$-irregular for $A_y^*$, the inverse of $A$ with respect to the sphere $S(y,1)$.\footnote{The concept of inner $\alpha$-thinness of a set at infinity, and particularly its alternative characterization provided by Theorem~\ref{th-th}(iii), has already found a number of applications to minimum $\alpha$-Riesz energy problems for condensers (see e.g.\ \cite{FZ-Pot1}, \cite{Z1}--\cite{Z-arx0}).}\end{definition}

\begin{remark}
     The concept of inner $\alpha$-thinness at infinity thus introduced actually coincides with that of $\alpha$-thinness at infinity by T.~Kurokawa and Y.~Mizuta \cite[Definition~3.1]{KM}. Indeed, to validate this, it is enough to show that the concept of capacity used in \cite{KM} (see p.~534 therein) is equivalent to that given by (\ref{capp}), which however directly follows from \cite[Theorem~4.2]{Z-arx-22}. Also note that in the case where $\alpha=2$ while $A$ is Borel, the above concept of inner $2$-thinness at infinity coincides with that of {\it outer\/} $2$-thinness at infinity, introduced by J.L.~Doob \cite[pp.~175--176]{Doob}.
\end{remark}

For the following theorem we refer to \cite[Theorem~2.3]{Z-bal2} (compare with \cite[Section~V.2.8]{L}, where $A$ was Borel).

\begin{theorem}\label{th-f-en}For arbitrary\/ $A\subset\mathbb R^n$, the following\/ {\rm(i$_1$)--(iii$_1$)} are equivalent.
\begin{itemize}
\item[{\rm(i$_1$)}] The inner\/ $\alpha$-Riesz capacity of\/ $A$ is finite:
\[c_\alpha(A)<\infty.\]
\item[{\rm(ii$_1$)}] For some\/ {\rm(}equivalently, every\/{\rm)} $y\in\mathbb R^n$,
\begin{equation*}\label{e-cap-f}\sum_{j\in\mathbb N}\,c_\alpha(A_j)<\infty,\end{equation*}
where\/ $q\in(1,\infty)$ and\/ $A_j:=A\cap\{x\in\mathbb R^n:\ q^j\leqslant|x-y|<q^{j+1}\}$.
\item[{\rm(iii$_1$)}] For some\/ {\rm(}equivalently, every\/{\rm)} $y\in\mathbb R^n$,
\begin{equation*}\label{bal-f-e}\varepsilon_y^{A_y^*}\in\mathcal E^+_\alpha.\end{equation*}
\end{itemize}
\end{theorem}

\begin{corollary}\label{twoC}
The following two conclusions\/ {\rm(a)} and\/ {\rm(b)} are obtained by comparing Theorems\/~{\rm\ref{th-th}} and\/ {\rm\ref{th-f-en}}.
\begin{itemize}
\item[{\rm(a)}] If\/ $A\subset\mathbb R^n$ is not\/ $\alpha$-thin at infinity, then necessarily\/ $c_\alpha(A)=\infty$.
\item[{\rm(b)}] There exists\/ $A_0\subset\mathbb R^n$ which is\/ $\alpha$-thin at infinity, but nonetheless, $c_\alpha(A_0)=\infty$.
\end{itemize}\end{corollary}

\begin{example}\label{exx}
For instance, the set $F_2\subset\mathbb R^3$, defined by means of formulae (\ref{descr}) and (\ref{ex2}), is $2$-thin at infinity; whereas  its Newtonian capacity is finite if and only if $s>1$, where $s$ is the parameter involved in (\ref{ex2}).
\end{example}

Using (\ref{k1}) and (\ref{har-eq}), we conclude from Theorem~\ref{th-f-en} that the inner $\alpha$-Riesz equilibrium measure $\gamma_A$ has finite energy if and only if $c_\alpha(A)<\infty$; and in the affirmative case it can alternatively be characterized, for instance, as the only measure in $\mathcal E'_A$ having the property $\kappa_\alpha\gamma_A=1$ n.e.\ $A$. Regarding the latter, see \cite[Theorem~9.1]{Z-arx-22}.

In the case where the inner equilibrium measure $\gamma_A$ still exists, although its energy may now be infinite, $\gamma_A$ can be described as follows (see \cite[Sections~5, 6]{Z-bal}).

\begin{theorem}\label{th-eqq} For any\/ $A\subset\mathbb R^n$ that is inner\/ $\alpha$-thin at infinity, the inner\/ $\alpha$-Riesz equilibrium measure $\gamma_A$ has the properties\/\footnote{Thus $\kappa_\alpha\gamma_A=1$ n.e.\ on $A$, see (\ref{KE}) and (\ref{onregu}).}
\begin{gather}\kappa_\alpha\gamma_A=1\text{ \ on\/ $A^r$},\label{onregu}\\
\kappa_\alpha\gamma_A\leqslant1\text{ \ on\/ $\mathbb R^n$}.\notag\end{gather}
Furthermore, $\gamma_A$ can be characterized as the unique measure in\/ $\mathfrak M^+$ satisfying either of the two limit relations
\begin{gather}\gamma_K\to\gamma_A\text{ \ vaguely in\/ $\mathfrak M^+$ as\/ $K\uparrow A$},\label{upar''}\\
\kappa_\alpha\gamma_K\uparrow\kappa_\alpha\gamma_A\text{ \ pointwise on\/ $\mathbb R^n$ as\/ $K\uparrow A$},\label{upar'}\end{gather}
where\/ $\gamma_K$ denotes the only measure in\/ $\mathcal E^+_K$ with\/ $\kappa_\alpha\gamma_K=1$ n.e.\ on\/ $K$.
\end{theorem}

\begin{remark}Alternatively,
(\ref{upar''}) and (\ref{upar'}) will follow from Theorem~\ref{bal-cont} once we show that for any given $Q\subset A$,
\begin{equation*}\label{baleq}
\gamma_Q=(\gamma_A)^Q.
\end{equation*}
To this end, it is enough to prove that $\Gamma_Q=\Gamma_{Q,\gamma_A}$ (cf.\ Definitions~\ref{def-bal}, \ref{def-eq}). This however is obvious in view of the fact that for any given $\nu\in\Gamma_{Q,\gamma_A}\cup\Gamma_Q$,
\[\kappa_\alpha\nu\geqslant1=\kappa_\alpha\gamma_A\text{ \ n.e.\ on $Q$},\]
which in turn is derived from (\ref{gamma}) (with $A:=Q$ and $\mu:=\gamma_A$), (\ref{def-in}), and (\ref{onregu}) by making use of Lemma~\ref{str}.
\end{remark}

\section{Proofs of Theorems~\ref{th1}, \ref{th-sharp}, and \ref{th-any}}\label{sec-proofs}

\subsection{Proof of Theorem~\ref{th1}} Given $\mu,\nu\in\mathfrak{M}^+$, assume that
\begin{equation}\label{eq1-th1}
  \kappa_\alpha\mu\leqslant\kappa_\alpha\nu\text{ \ n.e.\ on $A$},
  \end{equation}
where $A\subset\mathbb{R}^n$ is {\it not\/} inner $\alpha$-thin at infinity (Definition~\ref{def-th}). By Theorem~\ref{th-th}(iii),
\begin{equation}\label{mua}
  \mu^A(\mathbb{R}^n)=\mu(\mathbb{R}^n),
\end{equation}
$\mu^A$ being the inner balayage of $\mu$ to $A$. Noting from (\ref{eq1-th1}) that $\nu\in\Gamma_{A,\mu}$, where $\Gamma_{A,\mu}$ was introduced by (\ref{gamma}), we conclude from Definition~\ref{def-bal} that
\[\kappa_\alpha\mu^A\leqslant\kappa_\alpha\nu\text{ \ everywhere on $\mathbb R^n$},\]
and hence, by applying Theorem~\ref{th-fz}, that
\[\mu^A(\mathbb{R}^n)\leqslant\nu(\mathbb{R}^n).\]
Combining this with (\ref{mua}) proves $\mu(\mathbb{R}^n)\leqslant\nu(\mathbb{R}^n)$, which was the claim.

\subsection{Proof of Theorem~\ref{th-sharp}} Let $A\subset\mathbb{R}^n$ be inner $\alpha$-thin at infinity. Then, by virtue of Definition~\ref{def-th} and Theorem~\ref{th-th}(iii), one can choose $\mu_0\in\mathfrak{M}^+$ so that
\begin{equation}\label{mub}\mu_0^A(\mathbb{R}^n)<\mu_0(\mathbb{R}^n).\end{equation}
But according to (\ref{ineq1}) applied to $\mu_0$,
\[\kappa_\alpha\mu_0^A=\kappa_\alpha\mu_0\text{ \ n.e.\ on $A$},\]
which together with (\ref{mub}) validates the theorem with $\mu_0$, chosen above, and $\nu_0:=\mu_0^A$.

\subsection{Proof of Theorem~\ref{th-any}}\label{theend}  Given $A\subset\mathbb{R}^n$ and $c_\alpha$-absolutely continuous measures $\mu,\nu\in\mathfrak{M}^+_A$ having property (\ref{eq1-th1}), we aim to show that then necessarily \[\mu(\mathbb{R}^n)\leqslant\nu(\mathbb{R}^n).\]
The set $A^c$ being $\xi$-negligible for any $\xi\in\mathfrak{M}^+_A$, this is equivalent to the inequality
\begin{equation}\label{muc'}\int1_A\,d\mu\leqslant\int1_A\,d\nu,\end{equation}
where $1_A$ denotes the indicator function of $A$.

We can certainly assume that the set $A$ is inner $\alpha$-thin at infinity, for if not, the claim holds by virtue of Theorem~\ref{th1}. According to Theorem~\ref{th-th}, then there exists the inner $\alpha$-Riesz equilibrium measure $\gamma_A$, uniquely determined by Definition~\ref{def-eq}.

For any compact $K\subset A$, we obtain by Fubini's theorem
\begin{equation}\label{muc}\int\kappa_\alpha\gamma_K\,d\mu=\int\kappa_\alpha\mu\,d\gamma_K\leqslant\int\kappa_\alpha\nu\,d\gamma_K=
\int\kappa_\alpha\gamma_K\,d\nu,\end{equation}
where $\gamma_K\in\mathcal E^+_K$ denotes the equilibrium measure on $K$. (The inequality in (\ref{muc}) is valid in view of the fact that $\kappa_\alpha\mu\leqslant\kappa_\alpha\nu$ holds n.e.\ on $K$, hence $\gamma_K$-a.e., every Borel $E\subset\mathbb{R}^n$ with $c_\alpha(E)=0$ being $\xi$-neg\-lig\-ible for any $\xi\in\mathcal E^+_\alpha$, see Sect.~\ref{sec-prel'}.)

According to (\ref{upar'}), the net $(\kappa_\alpha\gamma_K)_{K\in\mathfrak C_A}$ of positive, lower semicontinuous functions increases pointwise on $\mathbb{R}^n$ to $\kappa_\alpha\gamma_A$. Applying \cite[Section~IV.1, Theorem~1]{B2} to the first and the last integrals in (\ref{muc}), we therefore get
\[\int\kappa_\alpha\gamma_A\,d\mu\leqslant\int\kappa_\alpha\gamma_A\,d\nu.\]
This implies (\ref{muc'}), because $\kappa_\alpha\gamma_A=1$ holds true n.e.\ on $A$ (see (\ref{onregu}) and (\ref{KE})), hence $(\mu+\nu)$-a.e. To verify the latter, observe that $N:=A\cap\{\kappa_\alpha\gamma_A<1\}$ is $(\mu+\nu)$-mea\-sur\-able, for so is the set $A$, $\mu+\nu$ being concentrated on $A$. Since $c_\alpha(N)=0$ while $\mu+\nu$ is $c_\alpha$-ab\-sol\-ut\-ely continuous, the set $N$ must be $(\mu+\nu)$-neg\-lig\-ible.

This completes the proof of Theorem~\ref{th-any}.
With regard to Remark~\ref{rem-any}, note that in the case $A\cap A_I=\varnothing$, the $c_\alpha$-absolute continuity of $\mu$ and $\nu$ is unnecessary for the validity of the above proof, for then, again by (\ref{onregu}), $\kappa_\alpha\gamma_A=1$ everywhere on $A$.

\section{Acknowledgements} The author is deeply indebted to Douglas P.\ Hardin and  Bent Fuglede for reading and commenting on the manuscript. 


\end{document}